\RequirePackage{fix-cm}
\documentclass{amsart}

\usepackage{graphicx}
\usepackage{mathrsfs}
\usepackage{amsfonts,amsmath}

\usepackage{color}
\usepackage{hyperref}

\renewcommand{\P}{{\mathbb P}}
\renewcommand{\L}{{\mathcal L}}

\newcommand{\mH}{{\mathbb H}}

\newcommand{\C}{{\mathbb C}}
\newcommand{\R}{{\mathbb R}}
\newcommand{\Q}{{\mathbb Q}}
\newcommand{\N}{{\mathbb N}}
\newcommand{\Z}{{\mathbb Z}}
\newcommand{\I}{{\mathbb I}}
\newcommand{\1}{{\mathbf 1 }}

\newcommand\Comb[2]{\left(\begin{array}{c}
#1\\
#2
\end{array}
\right)
}

\def\Aff{{\rm{Aff}}}

\newcommand{\E}{{\mathbb E}}

\newcommand{\T}{{\mathbb T}}

\newcommand{\leftB}{{{[\![}}}
\newcommand{\rightB}{{{]\!]}}}

\theoremstyle{plain}
\newtheorem{theorem}{Theorem}[section]
\newtheorem{lemma}[theorem]{Lemma}
\newtheorem{proposition}[theorem]{Proposition}

\theoremstyle{definition}
\newtheorem{remark}{Remark}[section]

\newcommand{\mysection}{\setcounter{equation}{0} \section}

\begin{document}

\title{The Brownian motion on \Aff($\R $) and Quasi-Local Theorems}
\author{V. Konakov}

\address{Higher School of Economics, National Research University, Shabolovka 31, Moscow, Russian Federation}

\email{VKonakov@hse.ru}

\author{ S. Menozzi}
\address{ Universit\'e d'Evry Val-d'Essonne, 23 Boulevard de France, 91037 Evry and Higher School of Economics, National Research University, Shabolovka 31, Moscow, Russian Federation}
\email{stephane.menozzi@univ-evry.fr}

\author{S. Molchanov}
\address{University of North Carolina at Charlotte, USA and Higher School of Economics, National Research University, Shabolovka 31, Moscow, Russian Federation}
\email{smolchan@uncc.edu}

\date{\today}

\thanks{The study has been funded by the Russian Science Foundation (project $ {\rm n}^{\circ} 17-11-01098$)}

\maketitle

\begin{abstract}
This paper is concerned with Random walk approximations of the Brownian motion on the Affine group $\Aff(\R) $. We are in particular interested in the case where the innovations are discrete. In this framework, the return probability of the walk have fractional exponential decay in large time, as opposed to the polynomial one of the continuous object. We prove that integrating those return probabilities on a suitable neighborhood of the origin, the expected polynomial decay is restored. This is what we call a \textit{Quasi-local theorem}.   
\end{abstract}

\mysection{Introduction}

In 
his seminal paper \cite{yor:92} (see also the related survey work \cite{yor:mats:05}), M. Yor studied the distribution 
density and the moments for the particular following exponential functional of the Brownian motion $(B_s)_{s\ge 0}$:
\begin{equation}
\label{DEF_A}
A_t^\nu=\int_0^t \exp(2B_s+\nu s) ds,
\end{equation}
which corresponds, up to a normalization in $t^{-1} $ to the quantity appearing in the Asian options in the Black and Scholes setting (see again \cite{yor:92}). The general case $\nu \neq 0$ can be reduced to $\nu=0 $ using the Girsanov theorem and the central object will be from now on the functional $A_t=\int_0^t \exp(2B_s) ds $.

This functional appears in several different situations, including the study of the Brownian motion on the group $\Aff(\R) $ of the affine transformations of $\R:x\mapsto ax+b, \ a,b\in \R,\ a>0$. 
This group can be isomorphically represented in the upper triangular $ 2\times 2$ matrices setting $g=\left[ \begin{array}{cc} a & b \\
0 & 1
\end{array}
\right],\ a>0 $. 
The affine group provides the simplest example of solvable Lie group. We announced several results on the Brownian motion $x_t:=\big(a_t,b_t\big) $ on  $\Aff(\R) $ in the short communication \cite{kona:meno:molc:11} which partly rely on the results by Yor \cite{yor:92}.

The central result of \cite{kona:meno:molc:11} is the following Theorem.
\begin{theorem}\label{THM1_KMM}
Let $p(t,\cdot,\cdot)$ be the transition density of the Brownian motion $x_t=\big(a_t,b_t\big) $ on $\Aff(\R) $ w.r.t. the corresponding Riemannian volume. Then, for all $g\in  \Aff(\R) $:
\begin{equation}
\label{DECAY_LARGE_TIME_DIAG}
p(t,g,g)=p(t,e,e)\sim_{t\rightarrow +\infty} \sqrt{\frac{\pi}{2}}\frac{1}{t^{\frac 32}},
\end{equation}
where $e=I $ is the neutral element of $\Aff(\R)$.
\end{theorem}
The note \cite{kona:meno:molc:11} contains similar results for other solvable Lie groups.
We will prove Theorem \ref{THM1_KMM} 
in Section \ref{SEC1}. 

The most interesting fact in Theorem \ref{THM1_KMM} is the slow decay of  $p(t,g,g),\ t\rightarrow +\infty$, which looks contradictory to the exponential growth of $\Aff(\R) $. Observe that such an exponential growth occurs for all non trivial finitely generated countable solvable groups, see e.g. Milnor \cite{miln:68}. 

We will establish that the random walks on the subgroups of $\Aff(\R)$ cannot directly give a \textit{good} approximation of the Brownian motion $(x_t)_{t\ge 0}$ on $\Aff(\R) $.

More precisely, if $(x_n^\varepsilon)_{n\in \N} $ is the Markov chain corresponding to a symmetric random walk on the subgroup $G^\varepsilon \subset G$ generated by the matrices $\left[ \begin{array}{cc}
\exp(\underline{+} \varepsilon) & 0 \\
0 & 1
\end{array}
\right]=g_{1}^{\underline{+} \varepsilon}; \left[ \begin{array}{cc}
1 & \underline{+} \varepsilon\\
0 & 1
\end{array}
\right]=g_{2 }^{\underline{+} \varepsilon}$ with step $\varepsilon^2 $ in time, $\varepsilon\in {\mathbb Q} $, then for $t=n\varepsilon^2\in \R^+$ we have that:
\begin{equation}\label{EXPO_DECAY}
 P^\varepsilon(t,g,g):=P^\varepsilon(n,g,g) =\P_g(x_n^\varepsilon=g)\le \exp(-cn^{\frac 13}\ln(n)^{\frac 23}),\ g\in G^\varepsilon.
 \end{equation}

 Let us stress that the exponential estimation $P^\varepsilon(n,g,g)\le \exp(-cn), c>0,$ which one could expect due to the exponential growth of the group cannot hold. Indeed, the solvable groups are amenable and it therefore follows from Kesten \cite{kest:59} that  $P^\varepsilon(n,g,g)$ decays at a sub-exponential rate.
We will establish in Section \ref{SEC3} by direct elementary arguments this estimate which is a particular case of the typical asymptotics obtained for the return probabilities of random walks on general solvable groups studied  e.g. by Pittet and Saloff-Coste \cite{pitt:salo:03} and Tessera \cite{tess:13}. The striking point is here that the return probability has fractional exponential decay and does not  behave as $\frac c{t^{\frac 32}} $
 as one could have expected from Theorem \ref{THM1_KMM}. The cause of this phenomenon is the special nature of the subgroup $G^\varepsilon$ (which is dense but again highly chaotically distributed).

Note that for the nilpotent groups, like e.g. the Heisenberg one ${\mathbf H}^3 $, the corresponding local limit theorems hold, see 
e.g. Breuillard \cite{breu:05} (like in the case of the random walk on $\Z^d$ see \cite{ibra:linn:71}, \cite{petr:05}, \cite{bhat:rao:76}). We also mention that for absolutely continuous innovations, a local Theorem on $\Aff(\R) $, with the expected rate in $n^{-3/2} $, matching the diagonal decay of the heat-kernel in \eqref{DECAY_LARGE_TIME_DIAG} for large times, was proved by Bougerol \cite{boug:83}. 

In this work, we will establish what we call \textit{quasi-local} theorems for the previously described random walk on the discrete subgroup.
Our first quasi-local theorem gives the estimation of the probability that $x^\varepsilon_n$ belongs to a small neighborhood of the unit element $e=I$ which shrinks to $e$ when $n\rightarrow +\infty $. We establish that the corresponding limit theorem holds with the expected convergence rate (see Section \ref{SEC4}).
It will be specified as well in Section \ref{SEC5} how such phenomena, i.e. the dramatic difference between the super-exponential decay of return probabilities stated in \eqref{EXPO_DECAY} and the polynomial one appearing when taking into account an associated neighborhood (which precisely corresponds to the large time behavior in \eqref{DECAY_LARGE_TIME_DIAG}), already occur
for a specific simple random walk on the dense locally uniformly distributed subgroup of $\R$ (generated by finitely many rationally independent numbers $\underline{+} \alpha_{i},\ i\in \{1,\cdots, N\} $). Roughly speaking, this dichotomy emphasizes that the paths of the random walk on the subgroup are quite \textit{dense}.
We will then eventually show that introducing (partially) an absolutely continuous component in the Markov chain $x_n^\varepsilon $ on $\Aff(\R)$, one can check that the densities of the finite dimensional distributions of $ x^\varepsilon_n$ converge uniformly to the corresponding densities of the diffusion on $\Aff(\R) $.

\mysection{Diffusion on $\Aff(\R) $ and similar groups}
\label{SEC1}
We briefly recall the construction of the Brownian motion on $\Aff(\R) $, see e.g. McKean \cite{mcke:69}, Ib\'ero \cite{TITI:76} or Rogers and Williams \cite{roge:will:85}. The Lie algebra $ \mathfrak{A}(\Aff(\R))$ consists of the matrices of the form $\left[\begin{array}{cc}
 x & y\\
 0 & 0\end{array}
 \right] $, $x,y \in \R $.  The metric on this algebra (i.e. in each plane of the tangent bundle of  $\Aff(\R) $) has the form $ds^2=dx^2+dy^2 $. The exponential mapping ${\rm Exp }$ from the algebra ${\mathfrak A}(\Aff(\R)) $ to the group $\Aff(\R) $  then writes:
\begin{equation}\label{EXP_MAP}
g={\rm Exp}\left(\left[ \begin{array}{cc}
x & y \\
0 & 0
\end{array}
\right]\right)=\left[ \begin{array}{cc}
a & b \\
0 & 1
\end{array}
\right]=\sum_{k\ge 0 }\frac 1 {k!}\left[ \begin{array}{cc}
x & y \\
0 & 0
\end{array}
\right]^k
=\left[ \begin{array}{cc}
e^x& e^x y \\
0 & 1
\end{array}
\right],
\end{equation}
i.e. $x=\ln(a), y=be^{-x}=\frac{b}{a}$ so that 
\begin{equation}
\label{metric}
ds^2=dx^2+dy^2=\frac{da^2+db^2}{a^2},
\end{equation}
i.e. the Riemannian metric on $\Aff(\R)$ is given by the same formula as the hyperbolic metric on the Poincar\'e model of the Lobachevskii plane (i.e. upper half plane of $\C $):
$$\C_+=\{b+ia, a>0 \}.$$
The ball of radius $R$ in this metric has an exponentially growing volume, i.e. $Vol(B(R))=2\pi \big(\cosh(R)-1\big)  $ (see e.g. Gruet \cite{grue:96}). 

In Section \ref{SEC3} we will consider the symmetric random walk on the finitely generated subgroups $G^\varepsilon \subset G $.
We consider the simplest subgroups with two generators:
\begin{eqnarray}
\label{DEF_SUB_GROUP}
\begin{array}{cc}
g_1^{\varepsilon}= \left[ \begin{array}{cc}
\exp( \varepsilon) & 0 \\
0 & 1
\end{array}
\right], &
g_2^{\varepsilon}= \left[ \begin{array}{cc}
1 & \varepsilon \\
0 & 1
\end{array}
\right] 
.
\end{array}
\end{eqnarray}
The number of different  words of length $n$ with the alphabet $\{ g_1^{\varepsilon},\ g_1^{-\varepsilon} , g_2^{\varepsilon}, g_2^{-\varepsilon} \} $ again grows exponentially with $n$ from Milnor \cite{miln:68} (non-niplotent or non-abelian solvable groups with finite number of generators have exponential growth).

The symmetric Brownian motion on $G$ can be constructed as the exponential mapping in the Stratonovich sense of the Brownian motion $
\big( B_t^1, B_t^2\big) $, i.e. $B^1,B^2 $ are two independent scalar Brownian motions, on  ${\mathfrak A}(\Aff(\R)) $:
\begin{equation}
\label{EQ_BM_AFF}
g_t=\left[ \begin{array}{cc}a_t & b_t \\
0 & 1
\end{array}
\right]= \big( \circ \big) \bigcap_{s=0}^t\left[ \begin{array}{cc}
(1+dB_s^1) & dB_s^2 \\
0 & 0
\end{array}
\right]=\left[ \begin{array}{cc}\exp(B_t^1) & \int_0^t \exp(B_s^1)dB_s^2 \\
0 & 1
\end{array}
\right].
\end{equation}
The generator of $\big(a_t,b_t\big)_{t\ge 0} $ writes for all $\varphi \in C^2(\R_+\backslash\{ 0\}\times \R,\R) $:
\begin{equation}
\label{GEN}
 L\varphi(a,b)=
 \frac 12\Big(a^2 \big(\partial_a^2+\partial_b^2\big)\varphi+a\partial_{a} \varphi \Big)(a,b)=:\Delta_{\Aff(\R)}\varphi(a,b),
 \end{equation}
where $\Delta_{\Aff(\R)} $ stands for the Laplace-Beltrami operator on $\Aff(\R) $.
Observe that the diffusion matrix $a^2I_2 $ is indeed the inverse of the Riemannian metric tensor $a^{-2}I_2 $.

To find the fundamental solution of the parabolic equation $\partial_t p=L p $, i.e. the transition density of the Brownian motion on $\Aff(\R)$, we will apply the Doob transform to the well known density of the Brownian diffusion on the hyperbolic space, see Karpelevich \textit{et al.} \cite{karp:tutu:shur:59} and Gruet \cite{grue:96} for multi-dimensional generalizations. We also refer to Bougerol \cite{boug:15} for other applications of Doob transforms on algebraic structures.

\begin{proposition}[Transition Density of the Brownian Motion on the hyperbolic plane $\mH^2$]
\label{PROP_DENS_HBM}
The density of the diffusion with generator
$$ \L\varphi(a,b)=\frac 12 a^2 \Delta \varphi(a,b) $$ 
w.r.t. the corresponding Riemann volume $\frac {dadb}{a^2} $ is given by:
\begin{equation}
\label{DEF_MBH}
p_{\mH^2}(t,x,y)=\frac{\sqrt 2\exp(-\frac t8)}{(2\pi t)^{3/2}}\int_r^{+\infty} \frac{u\exp(-\frac{u^2}{2t})}{\sqrt{\cosh(u)-\cosh(r)}} du,
\end{equation}
where $r=d_{\mH^2}(x,y)$ is the hyperbolic distance between $x=(a_1,b_1),y=(a_2,b_2) \in \mH^2 $, namely:
$$d_{\mH^2}(x,y)={\rm arcosh}\left(1+\frac{|x-y|^2}{2a_1a_2} \right),$$
where $|x-y|^2=|a_1-a_2|^2+|b_1-b_2|^2 $ is the usual squared Euclidean distance in $\R^2 $.
\end{proposition}
Now we want to use the Doob transform. The following Proposition holds, see e.g. Pinsky \cite{pins:95}.
\begin{proposition}[Doob transform]
\label{DT}
Let $M$ be a Riemannian manifold with metric $ds^2=g_{ij}dx^idx^j $ and corresponding Laplace-Beltrami operator
$$\Delta_M f(x)=\frac{1}{\sqrt{\det(g)}}\partial_{x_i}\Big(g^{ij}\sqrt{\det (g)}\partial_{x_j} f\Big)(x).$$ 
Let $p(t,x,y)$ be the fundamental solution of the heat equation $\partial_t p=\frac 12 \Delta_M p=-\frac 12 \Delta_M^* p   $ and $\psi(x)>0$ be the positive $\lambda $-harmonic function, i.e. it solves $\frac 12 \Delta_M \psi=\lambda \psi $. 
Put 
$$ p_\lambda(t,x,y)=\exp(-\lambda t)\frac{p(t,x,y)}{\psi(x)}\psi(y).$$
Then, $p_\lambda(t,x,y) $ is the transition density of a new diffusion on $M$ with generator:
$$ L_\lambda f(x)=\frac 12 \frac{\Delta_M (f\psi)(x)}{\psi(x)}-\lambda f(x)=\frac 12 \Delta_M f(x)+\nabla_M f (x) \cdot \nabla \ln(\psi(x)).$$
Here $\nabla_M$ stands for the Riemannian gradient, and the densities are always intended to be w.r.t. the corresponding Riemannian volume $\sqrt{\det g}dy$.  
\end{proposition}

Observe now that for $\psi(a,b)=a^{\frac 12} $, simple computations give that
$$\frac 12\Delta_{\mH^2}\psi(a,b)=\frac {a^2}{2}\big(a^{\frac 12}\big)'' =-\frac 18 \psi(a,b),\ \lambda=-\frac 18.$$
Combining Propositions \ref{PROP_DENS_HBM} and \ref{DT} and the above expression for $\psi $, we derive that the density of the Brownian motion on $\Aff(\R) $ can be expressed as the Doob-transform of the density of the Brownian motion on ${\mathbb H}^2 $.
\begin{theorem}[Density of the Brownian motion in $\Aff(\R)$ and Diagonal behavior in long time]\label{THM_ASYMP_DENS_GROUP}
The density $p_{\Aff}(t,e,\cdot) $ of the Brownian motion in $\Aff(\R)$ writes for all $(t,g,h)\in \R_+^*\times \Aff(\R)^2 $:
\begin{equation}
\label{DENS_AFF}
p_{\Aff(\R)}(t,g,h)=\exp\Big(\frac t 8\Big) \frac{p_{\mH^2}(t,g,h)}{\psi(g)} \psi(h)=\exp\Big(\frac t 8\Big) \frac{p_{\mH^2}(t,g,h)}{a^{\frac 12}} c^{\frac 12}, \ g=(a,b), h=(c,d).
\end{equation}
with $p_{\mH^2}$ as in \eqref{DEF_MBH}.
For $t\rightarrow +\infty$ one has for all $g\in \Aff(\R) $:
$$p_{\Aff(\R)}(t,g,g) \sim \frac{1}{(2\pi t )^{\frac 32}}\int_{0}^{+\infty}\frac{u}{\sinh(\frac u2)} du=\sqrt{\frac \pi 2}\frac{1}{t^{\frac 32}}.
$$ 
\end{theorem} 
The previous theorem has an important application in spectral theory (together with the remark that $p_{\Aff(\R)}(t,g,g)\sim \frac{C}{t} $ as $t\rightarrow 0 $, since $\dim(\Aff(\R))=2 $, see e.g. \cite{molc:75}).

\begin{theorem}
Consider on $\Aff(\R) $ the Schr\"odinger operator with non-positive fast decreasing potential $W(g)$:
 $$H=-\Delta_{\Aff(\R)} +W(g), \Delta_{\Aff(\R)}=\frac{1}{2}a^2\Big(\partial_a^2+\partial_b^2 \Big)+\frac 12 a\partial_a,$$
 and the spectral problem $H\psi=\lambda\psi $. Then, since operator $H$ has at most a finite negative spectrum $\{\lambda_j\le 0\} $, one has:
 $$N_0(W):=\sharp\{j: \lambda_j \le 0 \}\le C_1\int_{g\in \Aff(\R): 0\le |W(g)|\le 1} |W(g)|^{\frac 34}\sigma(dg)+C_2\int_{g\in \Aff(\R): |W(g)|>1} |W(g)|\sigma(dg),$$
 where for $g=(a,b),\ \sigma(dg)=\frac{dadb}{a^2} $ is the Riemannian volume element on $\Aff(\R) $. Also, the constants $C_1,C_2$ here are independent of the considered potential $W$ and can be computed directly. 
\end{theorem}
The previous Theorem is a direct consequence of the work by Molchanov  and Vainberg \cite{molc:vain:08}.\\

Eventually, we can also refer to Melzi \cite{melz:02} for a global upper bound of the density of the Brownian motion on $\Aff{\R} $. This work provides a tractable control for the diagonal and off-diagonal behavior of the heat-kernel in large time. 
\mysection{Approximation of Diffusion by Random Walks and Associated Return Probability Estimates}
\label{SEC3}

In this section, we are interested in the approximation of the Brownian motion on $\Aff(\R) $ by a discrete random walk. 
 Let now $\varepsilon $ be a given small parameter.
The time step of our random walk $(x_n^\varepsilon)_{n\ge 0}$ will be $\varepsilon^2 $ (with the usual parabolic scaling). In particular for a given time $ t>0 $,  it makes 
\begin{equation}
\label{SCALING}
n^\varepsilon(t)=\lfloor\frac{t}{\varepsilon^2} \rfloor
\end{equation}
steps on the interval $[0,t] $. 
Set $x_0^\varepsilon=\left[ \begin{array}{cc}
1 & 0\\
0 & 1
\end{array}
\right] $, and for all $n\ge 1$:
\begin{eqnarray*}
x_{n+1}^\varepsilon=x_{n}^\varepsilon A_{\varepsilon,n+1},\ A_{\varepsilon,n+1}= \left[ \begin{array}{cc}
\exp(\varepsilon X_{n+1}) & \varepsilon Y_{n+1}\\
0 & 1
\end{array}
\right],
\end{eqnarray*}
where the $(X_i)_{i\in \N^*}, (Y_i)_{i\in \N^*}$ are independent symmetric random variables, defined on some given probability space $(\Omega,{\mathcal A},\P) $, sharing the moment of the standard Gaussian law up to order two.
Hence, the above dynamics rewrites at time $n$:
\begin{eqnarray}
x_n^\varepsilon&:=&\left[ \begin{array}{cc}
a_n^\varepsilon & b_n^\varepsilon\\
0 & 1
\end{array}
\right]=\left[ \begin{array}{cc}
e^{\varepsilon \sum_{i=1}^n X_i} & \varepsilon \big(\sum_{i=1}^n Y_i \exp(\varepsilon\sum_{j=1}^{i-1}X_i) \big)\\
0 & 1
\end{array}
\right]\nonumber \\
&=:&\left[ \begin{array}{cc}
e^{\varepsilon S_n} & \varepsilon \big(\sum_{i=1}^n Y_i \exp(\varepsilon S_{i-1}) \big)\\
0 & 1
\end{array}
\right],\label{DEF_MARCHE}
\end{eqnarray} 
where we use the usual convention $\sum_{j=1}^0=0 $.
We will consider here mainly two cases.
\begin{trivlist}
\item[-] The Bernoulli Case: both $(X_i)_{i\in \N^*}, (Y_i)_{i\in \N^*} $ are independent sequences of independent Bernoulli random variables, i.e. $\P[X_1=1]=\P[X_1=-1]=\P[Y_1=1]=\P[Y_1=-1]=\frac 12 $. In such case, it is easy to see that the random walk stays on the subgroup $G^\varepsilon $.\footnote{Observe that this would as well be the case for any integer valued independent sequences $(X_i,Y_i)_{i\in \N^*} $ of independent random variables sharing the two first moments of the Gaussian law.} 
\item[-] The mixed case: $(X_i)_{i\in \N^*} , (Y_i)_{i\in \N^*} $ are independent sequences. The $(X_i)_{i\in \N^*}$ are still Bernoulli random variables whereas the $(Y_i)_{i\in \N^*} $ have an absolutely continuous law.
\end{trivlist}
In the first case we give an elementary proof that the return probability behave at least as $ \exp(-Cn^{\frac 13}\ln(n)^{\frac 23})$ for large $n$ (see \eqref{EXPO_DECAY} and Theorem \ref{THM_HK_BOUND}). We emphasize as well with the second case that, the density assumption for the $ (Y_i)_{i\in \N^*}$ is sufficient to restore the LLT (see Theorem \ref{LLT_2}).

\subsection{The Bernoulli Case.}
In this case, the idea is to express the non-diagonal element $b_n^\varepsilon$ in \eqref{DEF_MARCHE} in terms of the \textit{local times} $L(a,n) $ of the random walk $ (S_k)_{k\ge 0}$ at level $a\in [M_n^-,M_n^+] $, where 
$$M_n^-:=\min_{k\le n}S_k\le 0,\ M_n^+:=\max_{k\le n}S_k\ge 0. $$
We also precisely define:
$$L(n,a):=\sharp\{ k: S_k=a,\ 0<k\le n\}.$$
With these notations, we readily derive the following discrete \textit{occupation time formula}:
\begin{equation}
\label{OTF}
b_n^\varepsilon=\varepsilon \sum_{a=M_n^-}^{M_n^+} \big(\sum_{k\in \leftB 1,n\rightB: S_{k-1}=a} Y_k \big) \exp(\varepsilon a).
\end{equation}
The simplest (and yet very important) local theorem for $x_n^\varepsilon$ concerns the asymptotic behaviour of the return probability
$\pi_{2n}=P_e[x_{2n}^\varepsilon=e]=P[S_{2n}=0, \sum_{k=1}^{2n} Y_ke^{\varepsilon S_{k-1}}=0]$.

The exact asymptotic convergence rates of $\pi_{2n} $ can be found in \cite{pitt:salo:02}, \cite{pitt:salo:03}. Precisely, the following Theorem holds.
\begin{theorem}[Asymptotics of the return probabilities on the subgroup]\label{THM_HK_BOUND}
Assume that $ e^\varepsilon $ is transcendental. Then, there exists $c\ge 1$ s.t. for $n $  large enough:
$$c^{-1}n^{\frac 13} (\ln(n))^{\frac 23} \le -\ln (\pi_{2n})\le c n^{\frac 13} (\ln(n))^{\frac 23}.$$ 
\end{theorem}
In the quoted articles, the authors actually consider $\varepsilon=1 $, which readily gives the transcendence property.  In our work, we are interested in Donsker-Prokhorov type results (see Proposition \ref{DP} below), which will  require the previous scaling of \eqref{SCALING}. This leads us to consider the previous transcendence condition. Namely, if $e^\varepsilon $ is transcendent, and since $(S_i)_{i\in \N} $ is $\Z$ valued, we will have that: 
$$\sum_{i=1}^{2n} Y_i \exp(\varepsilon S_{i-1})=\sum_{a\in \leftB M_{2n}^-,M_{2n}^+ \rightB}\sum_{k\in \leftB 1,2n\rightB, S_{k-1}=a} Y_k \exp(\varepsilon a)=0\iff \forall a\in \leftB M_{2n}^-,M_{2n}^+\rightB,\ \sum_{k\in \leftB 1,2n\rightB, S_{k-1}=a} Y_k=0. $$ 

We now mention that, from the Lindemann-Weierstrass theorem, a sufficient condition for $e^\varepsilon $ to be transcendental is that $\varepsilon $ is algebraic, which for instance happens if $\varepsilon\in \Q $.\\

In the previously quoted article \cite{pitt:salo:03}, the lower bound of Theorem \ref{THM_HK_BOUND} follows from the Nash-Moser approach to heat kernel estimates. 
We now provide a proof for this lower bound, which relies on stochastic analysis arguments associated with some controls for the local time of the simple random walk, see e.g. \cite{reve:90}. 
\begin{proposition}
\label{BORNE_SUP}
If $e^{\varepsilon} $ is transcendental
then there exists $c\ge 1$  s.t. for $n$ large enough:
$$\pi_{2n}\le \exp(-c^{-1}n^{\frac 13}\ln(n)^{\frac 23}).$$
\end{proposition}
In particular, the proof emphasizes that the upper bound of the return probability does not depend on $\varepsilon $ as soon as it is algebraic.

\begin{proof}
The numbers $e^{k\varepsilon}, \ k\in \Z $ being rationally independent the probability $\pi_{2n} $ rewrites:
\begin{eqnarray}
\pi_{2n}&=& \P\Big[ 
\cap_{ a\in \leftB M_{2n-1}^-,M_{2n-1}^+\rightB
} 
L(2n-1,a)=0 \ {\rm Mod}\ 2Ê, S_{2n}=0,
 \forall a \in \leftB M_{2n-1}^-,M_{2n-1}^+\rightB \sum_{k\in \leftB 1,n\rightB:S_{k-1}=a}Y_k=0\Big].\notag\\
 \label{DEC_SPEC}
\end{eqnarray} 
Set now, $A:=\{\cap_{ a\in \leftB M_{2n-1}^-,M_{2n-1}^+\rightB} 
L(2n-1,a)=0 \ {\rm Mod}\ 2Ê, S_{2n}=0 \} $. We can thus write:
\begin{eqnarray}
\label{FIRST_BOUND}
\pi_{2n}= \E\left[
\prod_{a\in \leftB M_{2n-1}^-,M_{2n-1}^+\rightB}^{} 
\frac{\Comb{L(2n-1,a)}{ \frac{L(2n-1,a)}{2}}   }{2^{L(2n-1,a)}}\I_A\right]. 
\end{eqnarray}
Observe that, on the considered event $A$, for $a\in \leftB M_{2n-1}^-,M_{2n-1}^+\rightB $, the local time $L(2n-1,a) $ is even.
The contribution $\frac{\Comb{L(2n-1,a)}{ \frac{L(2n-1,a)}{2}}   }{2^{L(2n-1,a)}}$ then corresponds  to the probability that a symmetric Binomial law with parameter $L(2n-1,a) $ is equal to 0. This exactly describes the event $\sum_{k\in \leftB 1,n \rightB:S_{k-1}=a}Y_k=0 $. 

Observe importantly that on $A$:
$$\frac{\Comb{L(2n-1,a)}{\frac{L(2n-1,a)}{2}}   }{2^{L(2n-1,a)}}\le \frac 12.$$

Let us now localize w.r.t. the position of the minimum $M_{2n-1}^- $ and maximum ${M_{2n-1}^+}$. Namely, we want to get rid of the \textit{large deviations} for our current problem. Introduce the set $D_\alpha:=\{M_{2n-1}^-\le -\alpha\} \bigcup\{ 
 M_{2n-1}^+\ge \alpha \} $. Observe that
\begin{eqnarray*}
T_{2n}^{D_\alpha}&:=&\E\left[\prod_{a\in \leftB M_{2n-1}^-, M_{2n-1}^+ \rightB} \frac{\Comb{L(2n-1,a)}{\frac{L(2n-1,a)}{2}}   }{2^{L(2n-1,a)}}\I_{D_\alpha \cap A}\right]\le \left(\frac 12\right)^\alpha 2\P[M_{2n-1}^+\ge \alpha]\\
&\le& 4 \exp(- \alpha\ln 2) \exp(-\frac{\alpha^2}{4n}), 
\end{eqnarray*}
using the Bernstein inequality for the last control. Now in order to equilibrate the contributions of these \textit{large deviations} w.r.t the stated bound in Proposition \ref{BORNE_SUP} we want to solve the equation $\frac{\alpha^2}{n}+\alpha\ln 2= n^{\frac 13}\ln (n)^{\frac 23} $. It is then easily checked that the positive root $\alpha_n $ of the equation is s.t. $\alpha_n \sim_n \frac{n^{\frac 13}\ln(n)^{\frac 23}}{\ln(2)}=:m_n$. It thus follows that there exists $C_0^1$ s.t. for $n$ large enough:
$$T_{2n}^{D_{m_n}}\le \exp(-C_0^1m_n).$$  

On the other hand, we can as well derive the required control provided the extremas are \textit{small} with the previously emphasized  threshold. Namely, introducing:
\begin{equation}
\begin{split}
T_{2n}^S&:=\E\left[\prod_{a\in \leftB M_{2n-1}^-, M_{2n-1}^+ \rightB} \frac{\Comb{L(2n-1,a)}{ \frac{L(2n-1,a)}{2}}   }{2^{L(2n-1,a)}}\I_{|M_{2n-1}^-|\le \frac{m_n}{\ln(n)}, |M_{2n-1}^+|\le \frac{m_n}{\ln(n)}}\I_A\right]\\
&\le 
\P\Big[|M_{2n-1}^-|\le \frac{m_n}{\ln(n)}, M_{2n-1}^+\le \frac{m_n}{\ln(n)}, S_{2n=0}\Big]\le \P\Big[\forall k \in \leftB 0, 2n\rightB,\ \frac{S_k}{\sqrt n}\in [-\frac{m_n}{\sqrt n\ln(n)},\frac{m_n}{\sqrt n\ln(n)}]  \Big].
\end{split}
\label{TUBE}
\end{equation}
{\color{black}{
To control the last inequality we use the following important Lemma concerning \textit{tube estimates} for the random walk:
\begin{lemma}[Tube Estimates for the Random Walk]
\label{LEMME_TUBE}
There exists constants $c\le 1, C\ge 1$ s.t.:
\begin{equation*}
\begin{split}
\P[\forall k\in \leftB 1,2n\rightB, |S_k|\le \frac{m_n}{\ln(n)} ]\le C \exp(-c n^{\frac 13}\ln(n)^{\frac 23}),\\
\sum_{a=-m_n}^{m_n}\P[L(2n-1,a)>c^{-1}n^{\frac 23}\ln(n)^{\frac 13}]\le C \exp(-c n^{\frac 13}\ln(n)^{\frac 23}).
\end{split}
\end{equation*}
\end{lemma}
The above result can be viewed as a  discrete analogue of the tube estimates for the Brownian motion that can be found in \cite{iked:wata:80}. The proof is postponed to the end of the Section for the sake of clarity.

From Lemma \ref{LEMME_TUBE} and \eqref{TUBE} we get $T_{2n}^S\le  C\exp(-c n^{\frac 13}\ln(n)^{\frac 23})$. Thus}}, it suffices to restrict to the study of:
\begin{eqnarray*}
T_{2n}^M&:=&\E\Bigg[
\prod_{a\in \leftB M_{2n-1}^-, M_{2n-1}^+(2n-1) \rightB}
\frac{\Comb{L(2n-1,a)}{ \frac{L(2n-1,a)}{2}}   }{2^{L(2n-1,a)}}\I_{|M_{2n-1}^-|\le m_n ,M_{2n-1}^+\le m_n}\\
&&\times   \Big(\I_{M_{2n-1}^+> \frac{m_n}{\ln(n)}}+\I_{|M_{2n-1}^-|> \frac{m_n}{\ln(n)}}\Big)\I_A\Bigg].
\end{eqnarray*}
Fix now a $\delta\in (0,1) $ and introduce the random set:
$$A_\delta:=\{a\in \leftB M_{2n-1}^-, M_{2n-1}^+\rightB: L(2n-1,a)>n^\delta \}.$$
Let us now fix $c\in (0,1) $. If $\sharp A_\delta\ge c\frac {m_n}{\ln(n)}  $, then:
\begin{eqnarray*}
T_{2n}^{M,1}&:=&\E\Bigg[
\prod_{a\in ]M_{2n-1}^-,0[\cup]0, M_{2n-1}^+ [}
\frac{\Comb{L(2n-1,a)}{\frac{L(2n-1,a)}{2}}   }{2^{L(2n-1,a)}}\I_{|M_{2n-1}^-|\le m_n ,M_{2n-1}^+\le m_n}  \\
&& \times \Big(\I_{M_{2n-1}^+> \frac{m_n}{\ln(n)}}+\I_{|M_{2n-1}^-|> \frac{m_n}{\ln(n)}}\Big)\I_{\sharp A_\delta\ge c\frac {m_n}{\ln(n)}}\I_A\Bigg]\\
&\le &C \E[ \prod_{a\in A_\delta}  \frac{1}{L(2n-1,a)^{\frac 12}}  \I_{|M_{2n-1}^-|\le m_n ,M_{2n-1}^+\le m_n}  \Big(\I_{M_{2n-1}^+> \frac{m_n}{\ln(n)}}+\I_{|M_{2n-1}^-|> \frac{m_n}{\ln(n)}}\Big)\I_{\sharp A_\delta\ge c\frac {m_n}{\ln(n)}}\I_A]\\
&\le &  C(\frac 1{n^{\frac \delta 2}})^{c\frac{m_n}{\ln(n)}}=C\exp(-\frac\delta 2  \ln(n)\times c\frac{m_n}{\ln(n)} )=C\exp(-\frac\delta 2  c m_n ),
\end{eqnarray*}
where on the event $A_\delta $, we used the Stirling formula for the first inequality. 
It remains to handle:
\begin{eqnarray*}
T_{2n}^{M,2}&:=&\E\Bigg[\prod_{a\in \leftB M_{2n-1}^-, M_{2n-1}^+ \rightB} \frac{\Comb{L(2n-1,a)}{\frac{L(2n-1,a)}{2}}   }{2^{L(2n-1,a)}}\I_{|M_{2n-1}^-|\le m_n ,M_{2n-1}^+\le m_n}\\
&&  \Big(\I_{M_{2n-1}^+> \frac{m_n}{\ln(n)}}
+\I_{|M_{2n-1}^-|> \frac{m_n}{\ln(n)}}\Big)\I_{\sharp A_\delta< c\frac {m_n}{\ln(n)}}\I_A\Bigg].
\end{eqnarray*}
The first point to note is that, on the event $\{ \sharp A_\delta< c\frac {m_n}{\ln(n)}\}\cap \{ |M_{2n-1}^{-}|\le m_n,|M_{2n-1}^{+}|\le m_n\} $,  necessarily the occupation measure of $A_\delta$ is \textit{large}. Precisely, we have that defining:
$$ A_\delta^C:=\{a\in \leftB -m_n,m_n\rightB, a\not\in A_\delta\}, \ \sharp A_\delta^C\ge 2m_n-c\frac{m_n}{\ln(n)}.$$

On the other hand, the total \textit{local time} generated by the points in $A_\delta^C$ is less than $2m_n n^\delta=2n^{\frac 13+\delta}\ln(n)^{\frac 23}<n $, for $\delta\in (0,\frac 23)$ and $n$ large enough. Hence, the occupation time of $A_\delta $ is s.t.:
$$|\{i\in \leftB 1,2n\rightB: S_i\in A_\delta\}|>n.$$
Since we also know that on the considered event $\{ \sharp A_\delta< c\frac {m_n}{\ln(n)}\}$, we derive that there necessarily exists a level $a\in A_\delta $ s.t.
$$L(2n-1,a)>\frac{n}{c\frac{m_n}{\ln(n)}}.$$ 
We obtain:
\begin{eqnarray*}
T_{2n}^{M,2}&\le &\P[|\{i\in \leftB 1,2n\rightB: S_i\in A_\delta\}|>n,\sharp A_\delta< c\frac {m_n}{\ln(n)}, |M_{2n-1}^-|\le m_n ,M_{2n-1}^+\le m_n]\\
&\le  &\P[\exists 	a\in A_\delta, L(2n-1,a)>c^{-1}n^{\frac 23}\ln(n)^{\frac 13},\sharp A_\delta< c\frac {m_n}{\ln(n)}, |M_{2n-1}^-|\le m_n ,M_{2n-1}^+\le m_n]\\
&\le & \sum_{a=-m_n}^{m_n}\P[L(2n-1,a)>c^{-1}n^{\frac 23}\ln(n)^{\frac 13}]
\le
C \exp(-c n^{\frac 13}\ln(n)^{\frac 23}),
\end{eqnarray*}
using again Lemma \ref{LEMME_TUBE} for the last inequality.
\end{proof}

\begin{proof}[Proof of Lemma \ref{LEMME_TUBE} (Tubes for the random walk)]
Let us begin the proof observing that since,
\begin{eqnarray*}
\P[\forall k\in \leftB 1,2n\rightB, |S_k|\le \frac{m_n}{\ln(n)} ]\le \P[\exists a\in \leftB- \frac{m_n}{\ln(n)}, \frac{m_n}{\ln(n)}\rightB, L(2n,a)\ge \frac{n}{\frac{m_n}{\ln(n)}} ]\le \sum_{a=-\frac{m_n}{\ln(n)}}^{\frac{m_n}{\ln(n)}}\P[L(2n,a)\ge n^{\frac23}\ln ^{\frac 13}(n)],
\end{eqnarray*}
it suffices to prove the second statement of the Lemma. 
To this end, observe first that from Theorem 9.4 in Revesz \cite{reve:90}, we get for all $a>0, k\in \N $:
\begin{equation}
\label{EXPR_P_TL}
\P[L(2n,a)=k]=\begin{cases}
\frac{1}{2^{2n-k+1}}\Comb{2n-k+1}{(2n+a)/2}, \text{if}\ a\ \text{is even},\\
\frac{1}{2^{2n-k}}\Comb{2n-k}{(2n+a-1)/2}, \text{if}\ a\ \text{is odd}.
\end{cases}
\end{equation}
By symmetry we also derive that for $a<0$, the above expression holds replacing $a$ by  $|a| $ (recall indeed that $L(2n,a)\overset{({\rm law})}{=}L(2n,-a) $). Eventually, for $a=0$, Theorem 9.3 in \cite{reve:90} yields:
\begin{equation}
\label{EXPR_P_TL_0}
\P[L(2n,0)=k]=2^{-2n+k}\Comb{2n-k}{n}.
\end{equation}
Hence,
\begin{eqnarray*}
{\mathcal P}_{m_n}:=\sum_{a=-m_n}^{m_n}\P[L(2n,a)>c^{-1}n^{\frac 23}\ln(n)^{\frac 13}]=\P[L(2n,0)>c^{-1}n^{\frac 23}\ln(n)^{\frac 13}]+2 \sum_{a=1}^{m_n}\P[L(2n,a)>c^{-1}n^{\frac 23}\ln(n)^{\frac 13}].
\end{eqnarray*}
Note as well from \eqref{EXPR_P_TL} that,  in agreement with the intuition, $\P[L(2n,0)=k]>\P[L(2n,a)=k], a>0,k\in \N $. We therefore derive:
\begin{eqnarray*}
{\mathcal P}_{m_n} \le (1+2m_n)\P[L(2n,0)>c^{-1}n^{\frac 23}\ln(n)^{\frac 13}].
\end{eqnarray*}
Write now from \eqref{EXPR_P_TL_0}:
\begin{eqnarray}
\label{TO_BOUND}
{\mathcal P}_{m_n}\le (1+2m_n)\sum_{k=\lfloor c^{-1}n^{\frac 23}\ln(n)^{\frac 13} \rfloor}^{n} 2^{-2n+k}\Comb{2n-k}{n}.
\end{eqnarray}
By the Stirling formula, we obtain  that for $k\in \leftB \lfloor c^{-1}n^{\frac 23}\ln(n)^{\frac 13} \rfloor, n-1\rightB $,
\begin{equation}
\label{PROBA_BOUNDED_TL}
\P[L(2n,0)=k]=2^{-2n+k}\Comb{2n-k}{n} \le \frac{e}{ \pi\sqrt{ 2 n}}\frac{\sqrt{n-\frac k2}}{\sqrt{n-k}}\exp\left((2n-k)\ln(1-\frac{k}{2n})-(n-k)\ln(1-\frac kn) \right)  .
\end{equation}
The contribution for $k=n$ gives $\P[L(2n,0)=k]=2^{-n} $ and therefore a negligible term in the r.h.s. of \eqref{TO_BOUND}. We will now split the summation in \eqref{TO_BOUND} according to $k\in  \leftB \lfloor c^{-1}n^{\frac 23}\ln(n)^{\frac 13} \rfloor, n^{1-\eta}\rightB$ and $k\in  \leftB n^{1-\eta},n\rightB $ for $\eta >0 $ small enough to be specified later on. Observing that $\P[L(2n,0)=k] $ is a decreasing function of $k$ we obtain:
\begin{equation}
\label{PREAL_FIN_TUBE}
{\mathcal P}_{m_n}\le (1+2m_n)\le (1+2m_n) \bigg( \sum_{k=\lfloor c^{-1}n^{\frac 23}\ln(n)^{\frac 13} \rfloor}^{n^{1-\eta}}\P[L(2n,0)=k]  + n^\eta \P[L(2n,0)=n^{1-\eta}]\bigg).
\end{equation}
From \eqref{PROBA_BOUNDED_TL} it can be deduced from usual computations that there exists $C>0$ s.t. uniformly on $k\in \leftB \lfloor c^{-1}n^{\frac 23}\ln(n)^{\frac 13} \rfloor, n^{1-\eta}\rightB $, for $n$ large enough:
\begin{equation*}
\P[L(2n,0)=k]\le \frac{C}{\sqrt{n}}\exp\left(-\frac{k^2}{5n}\right).
\end{equation*}
Plugging this estimate in \eqref{PREAL_FIN_TUBE} yields:
\begin{eqnarray*}
{\mathcal P}_{m_n}&\le& C(1+2m_n)\bigg(\sum_{c^{-1}n^{\frac 16}\ln(n)^{\frac 13} <\frac{k}{\sqrt n}\le n^{1/2-\eta} } \frac{1}{\sqrt{2 \pi n}}\exp\left(-\frac 15 \Big(\frac k{\sqrt n} \Big)^2 \right)+n^{\eta-\frac 12} \exp\left(-\frac{n^{1-2\eta}}{5} \right ) \bigg)\\
&\le& C(1+2m_n)\left( \frac{1}{\sqrt {2\pi}}
\int_{c^{-1}n^{\frac 16}\ln(n)^{\frac 13}}^{+\infty}\exp(-\frac{x^2}{5})dx + \exp\left(-\frac{n^{1-2\eta}}{6}\right)\right)\\
&\le& C(1+2m_n) \left(\exp(-c^{-1}n^{\frac 13}\ln(n)^{\frac{2}{3}})+\exp\left(-\frac{n^{1-2\eta}}{6}\right)\right)\le C\exp(-c^{-1}n^{\frac 13}\ln(n)^{\frac{2}{3}}),
\end{eqnarray*}
taking $\eta \in(0,\frac 13)$ and up to modifications of $C,c$ for the last inequality.
This completes the proof.
\end{proof}


\mysection{Quasi-Local Theorems}
\label{SEC4}

We first mention that the integral theorem (which is an obvious corollary of the functional Donsker-Prokhorov  
Central Limit Theorem (CLT) for the random walks) of course applies. Namely, we have the following result.
\begin{proposition}[Donsker-Prokhorov approximation]\label{DP}
Fix $t>0$. If $\varepsilon\rightarrow 0 $, $n^\varepsilon(t):=\lfloor \frac t{\varepsilon^2}\rfloor\rightarrow +\infty  $, then
$$ (a_{\lfloor sn^{\varepsilon}(t)\rfloor}^\varepsilon,b_{\lfloor sn_\varepsilon(t)\rfloor}^\varepsilon)_{s\in [0,1]}\overset{({\rm law})}{\underset{\varepsilon\rightarrow 0}{\longrightarrow}} (a(st),b(st))_{s\in [0,1]},$$
where $a$ and $b $ are defined in \eqref{EQ_BM_AFF}. 
\end{proposition}
On the other hand, we are going to prove that some \textit{quasi}-local Theorems as well hold. By \textit{quasi}-local Theorem, we mean here that we consider a suitable renormalization of a neighborhood of the origin. Our main result in that direction is the following Theorem.

\begin{theorem}\label{LLT_1}
Let $g$ be a smooth test function s.t. its Fourier transform is compactly supported in $[-1,1] $ and s.t. $\int_\R g(x)dx=1 $. Denote, for a given $\delta>0$, by $g_\delta(x):=\frac 1 \delta g(\frac x \delta) $ its rescaling. Fix $t>0$, possibly large, and define for $n\in 2\N$, $\varepsilon_n=\left(\frac tn\right)^{\frac12} $.
Then, for  $\delta_n:=t^{\frac 12}n^{-\frac 12+\gamma},\gamma\in (0,\frac 12) $, we have:
\begin{equation}\label{QUASI_LLT_RATE}
\E\Bigg[ \I_{S_{n}=0} \ g_{\delta_n}\Big(\varepsilon_n \sum_{j=1}^{n} Y_j \exp(\varepsilon_n S_{j-1}) \Big)\Bigg] \sim_n \frac{2\varepsilon_n}{t^{\frac 12}\sqrt{2\pi} } 
\cdot p_2(t,0).
\end{equation}
Here, we denote for $t>0 $ by $p_2(t,\cdot) $ the density of the random variable $\tilde b_t:=\int_0^t e^{\tilde B^1_s}dB_s^2 $ where $\big( \tilde B_s^1 \big)_{s\in [0,t]} $ is a usual Brownian Bridge independent of the Brownian motion $B^2$. The subscript $2$ in $ p_2(t,\cdot)$, is here to recall the considered random variable is associated with the second component of the Brownian motion on the group.

Also,
\begin{equation}
\label{return_proba}
 p_2(t,0)=
\E\Big[ \frac{1}{\sqrt{2\pi \int_0^t e^{2\tilde B_s^1} ds}} \Big]\sim_{t\rightarrow +\infty} \frac{\pi}{t}, \ \frac{1}{\sqrt{2\pi t}} p_2(t,0)=p_{\Aff(\R)}(t,e,e).
\end{equation}
\end{theorem}
Hence, we find the expected  asymptotics in large time. We have a normalization in $\varepsilon_n $ 
and not in $\varepsilon_n^3 $ in \eqref{QUASI_LLT_RATE}, because we had already normalized our approximation of the stochastic integral in our scheme \eqref{DEF_MARCHE}.

We proceed to its proof in Section \ref{PROOF_LLT}.\\

To illustrate the phenomenon that appears on $\Aff(\R)$, i.e. the tremendous different rates between the pointwise return probabilities, and the \textit{quasi}-local Theorem, we consider a rather simple model which already enjoys such properties. 
Basically, this dichotomy emphasizes that, the discrete subgroups are in some sense \textit{very} dense, in the sense that they allow to have the expected convergence rates towards the densities of the limiting objects when integrated on a suitable neighborhood.

\subsection{Quasi-local CLT: the toy model} 
\label{SEC5}
We discuss in this section some points related to the local CLT on a dense subgroup $G_\varepsilon$ of a Lie group $G$ in the simplest possible case, taking $G=\R$, $G_1=\{ x: x=\sum_{i=1}^{N}n_i \alpha_i\} $ (or more generally $G_\varepsilon=\{ x: x=\varepsilon \sum_{i=1}^{N}d_i \alpha_i\},\ \varepsilon>0 $). Here, $N\in \N $ is a fixed given integer,  $\alpha=(\alpha_1,\cdots, \alpha_N) $ is s.t. the $\big\{ \alpha_i, i\in \{ 1,\cdots, N\}\big\} $  are rationally independent real numbers and  $d=(d_1,\cdots,d_N)\in \Z^N $ encodes the coordinates/displacements associated with the entries of $\alpha$.

The subgroup $G_1$ is not only dense in $\R$ but is also in some sense \textit{locally uniformly distributed}. This can for instance be seen from Herman Weyl's classical result (see e.g. \cite{stei:shak:03}). Consider for a fixed non negative integer $L$, the sequence 
$$\tilde x_{d}=\Big(\sum_{i=1}^N \alpha_i d_i\Big) \ {\rm Mod} \ L=\langle \alpha,d \rangle \ {\rm Mod} \ L,$$
where here the notation ${\rm Mod} \ L $ stands for the remainder term of the division by $L$. 
Then, for an arbitrary continuous and $L$ periodic function $f$ we have:
\begin{equation}
\label{PERIO}
\lim_{M\rightarrow +\infty} \frac{\sum_{d\in \Z^N: |d|\le M} f(\tilde x_d)}{\sharp\{d\in \Z^N: |d|\le M\}}=\frac{1}{L}\int_0^L f(x) dx,
\end{equation} 
where $|\cdot|$ stands here for the Euclidean norm of $\R^N $.

Consider now the symmetric random walk $ (x_n)_{n\in \N}$ on $\R $, s.t. $x_0=0 $, $x_n=\sum_{j=1}^n u_j $ where the $(u_j)_{j\in \N^*} $ are i.i.d. real-valued discrete random variables with law:
$$u_1\overset{({\rm law})}{=} p_0\delta_0 +\frac 12 \sum_{i=1}^N p_i(\delta_{\alpha_i}+\delta_{-\alpha_i}),\ \forall i\in \{1,\cdots, N\}, 0<p_i<1,\ \sum_{i=0}^N p_i=1.$$

We can as well consider the auxiliary random walk $(X_n)_{n\in \N}$ on $\R^N $ s.t. $X_0=0 $, $X_n=\sum_{j=1}^n U_j $ where the $(U_j)_{j\in \N^*} $ are i.i.d. $\R^N $-valued discrete random variables with law:
$$U_1\overset{({\rm law})}{=} p_0\delta_{0_{\R^N}} +\frac 12 \sum_{i=1}^N p_i(\delta_{\alpha_i e_i}+\delta_{-\alpha_i e_i}),\ \forall i\in \{1,\cdots, N\}, 0<p_i<1,\ \sum_{i=0}^N p_i=1.$$
In the above expression the $(e_i)_{i\in \{ 1,\cdots, N\}} $ denote the canonical basis vectors of $\R^N $.

Observe that the relation between the random variables $(u_j)_{j\in \N^*}$ and $((U_j)_{j\in \N^*}) $, and therefore between $x$ and $X$ is summarized as follows:
\begin{equation}
\label{REL_u_U_x_X}
\forall j\in \N^*,\ u_j=\langle U_j, \sum_{k=1}^N e_k\rangle=\langle U_j,\1\rangle,\ x_n=\langle X_n, \sum_{k=1}^N e_k \rangle=\langle X_n,\1\rangle,
\end{equation}
where $\1:=\sum_{k=1}^N e_k=(1,\cdots, 1)^* $.

Introduce now for notational convenience:
$$ \P[x_n=0]=r_n,$$
i.e. $r_n$ denotes the \textit{return} probability to $0$ at time $n$. We want to emphasize the following fact. Even though, from the standard CLT: 
\begin{equation}
\label{CLT}
 \frac{x_n}{\sqrt n} \underset{n}{\longrightarrow} {\mathcal N}(0,\sigma^2),\ \sigma^2=\E[u_1^2]=\sum_{i=1}^{N}p_i \alpha_i^2,
 \end{equation}
we \textbf{do not have} $r_n \sim_n \frac c{\sqrt n} $ but instead $r_n \sim_n \frac{c}{n^{N/2}} $. The result can be intuitively justified from the fact that from the rational independence of the  $\{ \alpha_i\}_{i\in \{1,\cdots,N\} }$, 
\begin{equation}
\label{EQ_P}
r_n=\P[x_n=0]=\P[X_n=0_{\R^N}].
\end{equation}
For the latter event, this means that in each direction the same number of positive and negative transitions are the same, and the asymptotics for this return probability corresponds to the product of the return probabilities in each direction. This fact can be formalized with the following proposition.

\begin{proposition}[Asymptotics for the return probability]
\label{AS_R_P}
As $n\rightarrow +\infty $, the following result holds:
\begin{trivlist}
\item[-] If $p_0>0 $, then:
$$r_n=\P[x_n=0] \sim_n \frac{C(p)}{n^{\frac N2}},\ C(p):=\prod_{i=1}^N \frac{1}{\sqrt{2\pi p_i 
}}.
$$
\item[-] If $p_0=0 $, then: $r_n=0$ if $n$ is odd and for $n$ even:
$$r_n=\P[x_n=0] \sim_n \frac{2C(p)}{n^{\frac N2}}.$$
\end{trivlist}
\end{proposition}
\begin{proof} Observe that $X_n$ is lattice valued. For a given $n\in \N$, defining ${\mathscr L}_n:=\{(\xi_1,\cdots, \xi_N),\ \forall i\in \{1,\cdots,N \}, \xi_i \in \{-n\alpha_i,\cdots,n\alpha_i\} \} $, we have $\P[X_n\in {\mathscr L_n}] =1$. Actually ${\rm supp}(X_n)\subset {\mathscr L}_n $, where the inclusion is strict. Write then for all $t\in \R^N$:
\begin{equation}
\label{TF_MARCHE}
\E[\exp(i \langle t,X_n\rangle)]=\sum_{\xi\in {\mathscr L}_n}\P[X_n=\xi] \exp(i \langle t,\xi\rangle ).
\end{equation}
Introducing the \textit{rescaled torus} $\T_N^\alpha:=
[-\frac{\pi}{\alpha_i},\frac{\pi}{\alpha_i}] 
$, we get that for all $(\xi,\zeta)\in {\mathscr L}_n , \frac{1}{|\T_N^\alpha|}\int_{\T_N^\alpha}\exp (-i\langle t,\xi\rangle) \exp(+i\langle t,\zeta\rangle) dt=\delta_{\xi,\zeta} $. Hence, for any $\xi_0\in {\rm supp}(X_n)$:
\begin{equation}
\label{INV_FOURIER_DIS}
\P[X_n=\xi_0]=\frac{1}{|\T_N^\alpha|}\int_{\T_N^\alpha}\exp(-i\langle t,\xi_0\rangle) \E[\exp(i \langle t,X_n\rangle)]dt=\frac{\prod_{j=1}^N \alpha_j}{(2\pi)^N}\int_{\T_N^\alpha}\exp(-i\langle t,\xi_0\rangle)\varphi^n(t)dt,
\end{equation}
where $\varphi(t):=\E[\exp(i\langle t , U_1 \rangle )]=p_0+\sum_{j=1}^N p_j\cos( t_j\alpha_j )=1+\sum_{j=1}^N p_j\big(\cos( t_j\alpha_j )-1\big)=1-2\sum_{j=1}^N p_j \sin^2\left(\frac{t_j \alpha_j}{2} \right)$.

Recalling \eqref{EQ_P}, we thus readily get from the inversion formula \eqref{INV_FOURIER_DIS} taking $\xi_0=0 $, and changing variable to $s_j=\alpha_jt_j,\ j\in \{1,\cdots,N\}$
\begin{eqnarray*}
r_n=\P[X_n=0_{\R^N}]=\frac{1}{(2\pi)^N}\int_{\T_N} \varphi^n\Big(\frac{s_1}{\alpha_1} ,\cdots,\frac{s_N}{\alpha_N}\Big)ds=\frac{1}{(2\pi)^N}\int_{\T_N} \Big(1-2\sum_{j=1}^N p_j \sin^2\left(\frac{s_j}{2}\right)\Big)^n ds ,
\end{eqnarray*}
where $\T_N:=[-\pi,\pi]^N$.
For  \textit{small values} of $|s|  $ we then get that:
\begin{equation}
\label{asymp_s_petit}
\varphi\Big(\frac{s_1}{\alpha_1} ,\cdots,\frac{s_N}{\alpha_N}\Big)=1+\sum_{j=1}^Np_j\Big(-\frac{s_j^2}{2}+O(s_j^3)\Big) =\exp\Big(-\frac 12\sum_{j=1}^N p_js_j^2+O(|s|^3)\Big).
\end{equation}
Set $\delta_n:=c\left(\frac{\ln(n)}{ n}\right)^{1/2},\ c>\left(\frac{N}{\min_{j\in \{1,\cdots,N\}}p_j}\right)^{1/2} $.
We now introduce $ B_N(\delta_n):=\{ s\in \T_N: |s|_\infty\le \delta_n\}$ (ball of radius $\delta_n$ around the origin) and $C_N(\delta_n):=\{ s\in \T_N: \forall j\in \leftB 1,N\rightB,\ s_j\in [-\pi,-\pi+\delta_n]\cup[\pi-\delta_n,\pi] \} $ (corners of radius $\delta_n $ of the torus $\T_N $). Set $M_N(\delta_n):=B_N(\delta_n)\cup C_N(\delta_n) $.
 Observe that for $s\in \T_N\backslash M_N(\delta_n)$, 
we have either:
\begin{trivlist}
\item[(a)]$\exists j_0\in \{ 1,\cdots,N\},\ \cos(s_{j_0})-1=-2\sin^2(\frac{s_{j_0}}2) \in [-2+\frac{\delta_n^2}{2}+o(\delta_n^2),-\frac{\delta_n^2}{2}+o(\delta_n^2)].$
\item[(b)] $K_S:=\{j\in \{1,\cdots,N\}: |s_j|\le \delta_n \}$ and $K_{L}:=\{j\in \{1,\cdots,N\}: (\pi-|s_j|)\le \delta_n \} $ are non empty.
\end{trivlist}
In case (a), we readily get $|1-2\sum_{j=1}^N p_j \sin^2\left(\frac{s_j}{2}\right)|\le \left(1-p_{j_0}\frac{\delta_n^2}{2}+o(\delta_n^2)\right)$. In case (b), we derive:
$$|1-2\sum_{j=1}^N p_j \sin^2\left(\frac{s_j}{2}\right)|\le |1-2\sum_{k\in K_L}p_k|+ \frac{\delta_n^2}2+o(\delta_n^2):=c_{L,S}(n)\le 1-\frac 12 \min_{j\in \{1,\cdots,N\}}p_j,$$
for $n $ large enough.
We can therefore rewrite:
\begin{eqnarray*}
r_n&=&
\frac{\1}{(2\pi)^N}\int_{M_N (\delta_n)} \Big(1-2\sum_{j=1}^Np_j\sin^2\left(\frac{s_j}2\right)\Big)^n ds+R_N^n,\\
|R_N^n|&\le &C\int_{\T_N\backslash M_N(\delta_n)}\left\{\left(1-p_{j_0}\frac{\delta_n^2}{2}+o(\delta_n^2)\right)^n +c_{L,S}(n)^n\right\}ds\le Cn^{-\frac{p_{j_0}c^2}{2}}=o(n^{-N/2}).
\end{eqnarray*}
Let us discuss now the contribution associated with $C_N(\delta_n) $. For $s\in C_N(\delta_n) $, one has for all $j\in \{1,\cdots,N\} $:
$$-2 \sin^2\left(\frac{s_j}2\right)= -2\left(1-\Big(\frac{\pi-|s_j|}2\Big)^2\right)+O\big((\pi-|s_j|)^3\big),$$
so that $1-2\sum_{i=1}^N p_j\sin^2\left(\frac{s_j}2 \right) =-1+2p_0+\sum_{j=1}^N p_j\frac{(\pi-|s_j|)^2}{2} +O\big((\pi-|s_j|)^3\big)$. Hence, 
\begin{trivlist}
\item[-] if $p_0\neq 0$, we thus readily get $\frac{\1}{(2\pi)^N}\int_{C_N (\delta_n)} \Big(1-2\sum_{j=1}^Np_j\sin^2\left(\frac{s_j}2\right)\Big)^n ds=o(n^{-N/2})$.
\item[-] if $p_0=0$, by symmetry, we get $r_n=0 $ if $n$ is odd and
\begin{eqnarray*}
\frac{1}{(2\pi)^N}\int_{M_N (\delta_n)} \Big(1-2\sum_{j=1}^Np_j\sin^2\left(\frac{s_j}2\right)\Big)^n ds
= \frac{2}{(2\pi)^N}\int_{B_N (\delta_n)} \Big(1-2\sum_{j=1}^Np_j\sin^2\left(\frac{s_j}2\right)\Big)^n ds,
\end{eqnarray*}
if $n$ is even. Recall now from \eqref{asymp_s_petit} that:
\begin{eqnarray*}
\frac{1}{(2\pi)^N}\int_{B_N (\delta_n)} \Big(1-2\sum_{j=1}^Np_j\sin^2\left(\frac{s_j}2\right)\Big)^n ds
=\frac{1}{(2\pi)^N}\int_{B_N (\delta_n)} \exp\Big(-n \Big\{\sum_{j=1}^Np_j\frac{s_j^2}2+O(|s|^3)\Big\}\Big) ds\\
=\frac{1}{(2\pi n)^{\frac N2}\prod_{j=1}^N \sqrt{p_j}} \int_{\prod_{j=1}^N \Big\{|\tilde s_j|\le \ln(n)^{1/2}p_j^{1/2}\Big\}}\exp\Big(-\frac 12 \sum_{j=1}^N \tilde s_j^2 +O\big( \frac{\ln(n)^{3/2}}{n^{1/2}})\Big)  \frac{d\tilde s}{(2\pi)^{\frac N2}}\sim_n \frac 1{n^{\frac N2}}\prod_{i=1}^N \frac{1}{\sqrt{2\pi p_i 
}} =\frac{C(p)}{n^{\frac N2}}.
\end{eqnarray*}
This gives the stated result.
\end{trivlist}
We can as well refer more generally to the proof of the classical local CLT (see e.g. \cite{petr:05} or Chapter 5 in \cite{bhat:rao:76} for the multidimensional case).

Observe that the asymptotic of the return probability $r_n$ does not depend on the rationally independent numbers $(\alpha_j)_{j\in \{1,\cdots,N\}} $ chosen. We simply used the fact that, to return to 0, we must have over the considered time interval, for all $j\in \{1,\cdots ,N\} $, the same numbers of random variables taking the values $-\alpha_j e_j $ and $\alpha_j e_j $.
\end{proof}
Hence, the bigger $N$, the smaller the \textit{exact} return probability. Similarly, 
from \eqref{INV_FOURIER_DIS} we can extend the previous proposition with the following result.
\begin{proposition}[Deviation bounds for the LLT]\label{PROP_DEV}
Let $n\rightarrow +\infty $ and  $y\in \R^N\cap {\rm supp}(X_n)$ be s.t. its Euclidean norm $|y|\le n^{\frac 23-\gamma},\ \gamma>0 $ (which is meant to be small). Then, for $p_0>0$, recalling as well that $X_0=0 $, we obtain:
$$\P[X_n=y]\sim_n \prod_{j=1}^N  \left(\frac{\exp(-\frac{y_j^2}{2\alpha_j^2 p_j n})}{(2\pi p_j 
n)^{\frac 12}} \right).$$
\end{proposition}
\begin{proof}
We indicate that starting from \eqref{INV_FOURIER_DIS}, proceeding as in the previous proof of Proposition \ref{AS_R_P} and considering a localization with respect to a ball of radius $\delta_n=n^{-1/3+\gamma/3}$, we derive:
\begin{eqnarray*}
\P[X_n=y]&=&\frac{1}{(2\pi)^N}\int_{B(\delta_n)} \exp\Big(-i\left\langle \Big(\frac{y_1}{\alpha_1},\cdots,\frac{y_N}{\alpha_N}\Big), s\right\rangle\Big)\exp(-\frac n2\sum_{j=1}^N p_j s_j^2)ds+R_N^n=:M_N^n+R_N^n,\\
|R_N^n|&\le & C \int_{\T_N\setminus B(\delta_n)} \Big((1-p_{j_0}\frac{\delta_n^2}{2})^n+c_{L,S}(n)^n\Big) ds\le C\exp(-cn^{1/3+2\gamma/3}).
\end{eqnarray*}
On the one hand, it is clear that $M_N^n\sim_n \prod_{j=1}^N  \left(\frac{\exp(-\frac{y_j^2}{2\alpha_j^2 p_j n})}{(2\pi p_j 
n)^{\frac 12}} \right)$.
On the other hand, on the considered range set for $y$, we have that $|R_N^n|\le C\exp(-cn^{1/3+2\gamma/3}) $. Hence, this term can indeed be seen as a global remainder uniformly in $y$. This yields the result.
Observe as well that for $y\in {\rm supp}(X_n) $, $\Big(\frac{y_1}{\alpha_1},\cdots, \frac{y_N}{\alpha_N})\in \Z^N $.
\end{proof}

Observe now that from the previous definition of $x_n $, for any $\Gamma\subset \R $, 
 \begin{equation}
\label{SOMMATION}
 \P[x_n\in \Gamma]=\P[\langle X_n,\1 \rangle \in \Gamma]=\sum_{y\in \Z^N, \langle y,\alpha\rangle \in \Gamma}\P[X_n=\sum_{i=1}^N \alpha_i y_i e_i] .
 \end{equation}
From equation \eqref{SOMMATION} in Proposition \ref{PROP_DEV}, we derive the following theorem.
\begin{theorem}\label{CTR_THEO_DEV}
 For a given $\gamma \in (0,\frac 12) $,  and a positive sequence $\delta_n \rightarrow_n 0$ and s.t. $\delta_n \ge n^{-(\frac 12-\gamma)}$, we have for $p_0>0$:
 $$P[x_{2n}\in (-\delta_n^{-1},\delta_{n}^{-1})] \sim_n 2 \delta_n^{-1} \frac{1}{\sqrt{2\pi (2n)} \sigma} 
 ,$$
 where  as in the usual CLT stated in \eqref{CLT}, $\sigma^2=\sum_{i=1}^N p_i\alpha_i^2 
$.
\end{theorem}
From Proposition \ref{AS_R_P} and Theorem \ref{CTR_THEO_DEV}, we precisely see that, the integrated probability gives the expected usual rate in $n^{-1/2} $. Actually, this is precisely due to the last part of Proposition \ref{PROP_DEV}, we integrate in a neighborhood of a hyperplane of $\R^N$, whereas the pointwise return probabilities might have arbitrarily polynomial decay in function of the chosen $N$. We will now show a similar behavior for our random walk on $\Aff(\R)$.

\subsection{Proof of Theorem \ref{LLT_1}}
We first need the following auxiliary lemma concerning the maximum of the conditioned random walk.
\begin{lemma}[Maximum and Minimum of the conditioned random walk]
\label{LEMME_BRIDGE_WALK}
Let $n\ge 0$ be given and consider the conditioned random walk $\big(\tilde S_j \big)_{j\in \leftB 0,n\rightB} $, $\tilde S_j=\sum_{i=1}^j X_i$ s.t. $\tilde S_0=\tilde S_n=0 $. We recall here that $(X_i)_{i\in \N^*}$ is a sequence of i.i.d. Bernoulli random variables.  Denoting by $\tilde M_n^+:=\max_{i\in \leftB 0,n\rightB}\tilde S_i,\tilde M_n^-:=\min_{i\in \leftB 0,n\rightB} \tilde S_i$ we have that for all $\theta>0 $ there exists $c:=c(\theta)\ge 1$ s.t. 
\begin{equation}
\label{DEV_BB_WALK}
 \E\Big[\exp\Big(\theta \frac{\tilde M_n^+}{\sqrt n}\Big)\Big]+\E\Big[\exp\Big(\theta \frac{\tilde M_n^-}{\sqrt n}\Big)\Big]\le c\exp( c\theta^2).
\end{equation}
\end{lemma}
\begin{proof}
It is well known from the Donsker invariance principle that $\tilde M_n^+,\tilde M_n^- $ respectively converge in  law towards the maximum and the minimum of a standard Brownian bridge on $[0,1]$ (see e.g. Liggett \cite{ligg:68} or Vervaat \cite{verv:79}). For the rest of the proof we focus on $\tilde M_n^+ $, the results for $\tilde M_n^- $ can be derived similarly by symmetry.

For any $A>0$, denoting by $\tilde M^+:=\sup_{s\in [0,1]}\tilde B_s  $ where $\big( \tilde B\big) _{s\in [0,1]} $ is a standard Brownian bridge, we get that for all $\theta\ge 0 $:
$$\E\Bigg[\exp\Big(\theta \frac{\tilde M_n^+}{\sqrt{n}}\Big)\I_{\left|\frac{\tilde M_n^+}{\sqrt{n}}\right|\le A}\Bigg]\underset{n}\longrightarrow \E\Bigg[\exp(\theta \tilde M^+)\I_{\left|\tilde M^+\right|\le A}\Bigg]\le  \E\Big[\exp(\theta \tilde M^+)\Big].$$
Letting $A\rightarrow \infty$, we then obtain by 
usual uniform integrability arguments that:
$$\E\Bigg[\exp\Big(\theta \frac{\tilde M_n^+}{\sqrt{n}}\Big)\Bigg] \underset{n}{\longrightarrow} \E\Big[\exp(\theta \tilde M^+)\Big].$$
Therefore, there exists $C:=C(\theta)\ge 1$ s.t. for all $n\ge 0$,
$$\E\Bigg[\exp\Big(\theta \frac{\tilde M_n^+}{\sqrt{n}}\Big)\Bigg] \le C\E\Big[\exp(\theta \tilde M^+)\Big]\le C\exp(c\theta^2),$$
where the last inequality simply follows from the exact expression of the joint law of the Brownian motion and its running maximum, see e.g. \cite{revu:yor:99}. 
\end{proof}

\label{PROOF_LLT}
\begin{proof}[Proof of Theorem \ref{LLT_1}]
We have first, for even $n$:
\begin{eqnarray*}
\E\Bigg[ \I_{S_{n}=0} \ g_{\delta_n}\Big(\varepsilon_n \sum_{j=1}^{n} Y_j \exp(\varepsilon_n S_{j-1}) \Big)\Bigg]&=&
\P[S_n=0] \E \Bigg[g_{\delta_n}\Big(\varepsilon_n \sum_{j=1}^{n} Y_j \exp(\varepsilon_n S_{j-1}) \Big)|S_n=0\Bigg]\\
&=&\P[S_n=0] \E \Bigg[g_{\delta_n}\Big(\varepsilon_n \sum_{j=1}^{n} Y_j \exp(\varepsilon_n \tilde S_{j-1}) \Big)\Bigg],
\end{eqnarray*}
where $(\tilde S_j)_{j\in \leftB 1,n\rightB} $ stands for the random walk conditioned to be at 0 at time $n$. Then:
\begin{eqnarray*}
\E\Bigg[ \I_{S_{n}=0} \ g_{\delta_n}\Big(\varepsilon_n \sum_{j=1}^{n} Y_j \exp(\varepsilon_n S_{j-1}) \Big)\Bigg]
&\sim_n &
\frac{2}{\sqrt{2\pi n}}
\E\Bigg[\frac 1{2\pi}\int_\R \hat g(\delta_n x)\exp\Big(-i\varepsilon_n x \sum_{j=1}^nY_j \exp(\varepsilon_n \tilde S_{j-1}) \Big)dx\Bigg].
\end{eqnarray*}
Taking the conditional expectation w.r.t. to $(\tilde S_{j})_{j\in \leftB 1 ,n\rightB} $ and using the symmetry of the i.i.d random variables $(Y_j)_{j\in \leftB 1,n\rightB} $, we derive:
\begin{eqnarray*}
\E\Bigg[ \I_{S_{n}=0} \ g_{\delta_n}\Big(\varepsilon_n \sum_{j=1}^{n} Y_j \exp(\varepsilon_n S_{j-1}) \Big)\Bigg]
&\sim_n &
\frac{2}{\sqrt{2\pi n}}
\frac 1{2\pi}\int_\R \hat g(\delta_n x)\E\Bigg[\prod_{j=1}^n\cos\big(\varepsilon_n x \exp(\varepsilon_n \tilde S_{j-1})\big) \Bigg]dx.
\end{eqnarray*}
Let now $\widetilde M_n^{-}, \widetilde M_n^{+} $ denote the respective minimum and maximum values of the conditioned random walk (bridge) $(\tilde S_j)_{j\in \leftB 1,n\rightB} $. We can assume w.l.o.g. that $|\widetilde M_n^{\underline{+}}|\le cn^{\frac 12 } \ln(n)^{\frac 12} $ for a sufficiently large constant $c$. Indeed, 
\textcolor{black}{
\begin{eqnarray*}
\P[|\widetilde M_n^{\underline{+}}|\ge cn^{\frac 12 } \ln(n)^{\frac 12}]&=&\frac{\E[ \I_{|M_n^{\underline{+}}|\ge cn^{\frac 12}\ln(n)^{\frac 12}} \I_{S_n=0}] }{\P[S_n=0]}\\
&\le& Cn^{\frac 12}\P[|M_n^{\underline{+}}|\ge cn^{\frac 12}\ln(n)^{\frac 12}]^{\frac 1p}\P[{S_n=0}]^{\frac 1q}, \ p,q>1, \frac 1p+\frac 1q=1,
\end{eqnarray*}
using the lower bound of the control $$\frac{C^{-1}}{\sqrt n}\le   \P[S_n=0]=\Comb{n}{n/2}\frac{1}{2^n}\le \frac{C}{\sqrt n}, \ C\ge 1,$$
which follows from the Stirling formula, for the last inequality.
The upper bound and the Bernstein inequality\footnote{We can also refer here to formula (2.16) of Theorem 2.13 in \cite{reve:90} for a more precise result which is not needed for our current purpose.} for the standard random walk on $\Z$ then yield:
\begin{eqnarray*}
\P[|\widetilde M_n^{\underline{+}}|\ge cn^{\frac 12 } \ln(n)^{\frac 12}]&\le & C \exp\left(-\frac {c^2}{2p} \ln(n)\right)n^{\frac12(1-\frac 1q)}=Cn^{-\frac{c^2}{2p}+\frac 12(1-\frac 1q)},
\end{eqnarray*}
which again gives a negligible contribution w.r.t. to the scale $n^{-\frac 32}$ for $c$ large enough. 
}

Recalling as well that we have assumed $\hat g $ to be compactly supported in $[-1,1] $, we get that we only have to consider the integration variable $x$ in the range $|x|\le \frac{1}{\delta_n} $. Recall from the statement of Theorem \ref{LLT_1} that $\frac{\varepsilon_n}{\delta_n} = n^{-\gamma}$ for $0<\gamma<\frac 12 $.
Then, for all $j\in \leftB 1, n \rightB $, on the event $\{\widetilde M_n^+\le c n^{\frac 12}\ln (n)^{\frac 12} \} $: 
\begin{equation}
\label{NEW_EQ_FOR_REMAINDERS}
\varepsilon_n |x|\exp(\varepsilon_n \tilde S_{j-1})\le \frac{\varepsilon_n}{\delta_n}\exp(\varepsilon_n \tilde S_{j-1})\le n^{-\gamma}\exp(\varepsilon_n \widetilde M_n^+)\le n^{-\gamma}\exp\Big(ct^{\frac 12}\big(\ln(n)\big)^{\frac 12}\Big)\rightarrow_n 0.
\end{equation}
On the associated sets, we will therefore obtain that the arguments in the cosines are uniformly small. Precisely:
\begin{eqnarray*}
\E\Bigg[\prod_{j=1}^n\cos\big(\varepsilon_n x Y_j\exp(\varepsilon_n \tilde S_{j-1})\big) \I_{\widetilde M_n^+\le cn^{\frac 12 }\ln(n)^{\frac 12}}\Bigg]\\
=
\E\Bigg[\prod_{j=1}^n\Bigg(1-\frac{(\varepsilon_n x)^2 \exp(2\varepsilon_n \tilde S_{j-1}) }2
+O\Big((\varepsilon_n x)^4 \exp(4\varepsilon_n \tilde S_{j-1})
\Big) \Bigg)
\times \I_{\widetilde M_n^+\le cn^{\frac 12 }\ln(n)^{\frac 12}}\Bigg]\\
=\E\Bigg[\exp\Bigg(-\sum_{j=1}^n\Big\{\frac{(\varepsilon_n x)^2 \exp(2\varepsilon_n \tilde S_{j-1}) }2+
O\Big((\varepsilon_n x)^4 \exp(4\varepsilon_n \tilde S_{j-1})
\Big) \Big\}\Bigg)
\times \I_{\widetilde M_n^+\le cn^{\frac 12 }\ln(n)^{\frac 12}}\Bigg].
\end{eqnarray*}
Hence,
\begin{equation}
\begin{split}
\frac 1{2\pi}\int \hat g(\delta_n x)\E\Bigg[\prod_{j=1}^n\cos\big(\varepsilon_n x Y_j\exp(\varepsilon_n \tilde S_{j-1})\big) \Bigg]dx\\
\sim_n \frac 1{2\pi} \int \hat g(\delta_n x) \E\Bigg[\exp\Bigg(-\frac{x^2}{2} (\tilde A_n(t)+\tilde R_n(t))\Bigg)\I_{\widetilde M_n^+\le cn^{\frac 12 }\ln(n)^{\frac 12}}\Bigg]dx=:I_n, \label{DEF_IN}
\end{split}
\end{equation}
where,
\begin{eqnarray}
\label{DECOMP_AN}
\tilde A_n(t):=
\varepsilon_n^2\sum_{j=1}^n\exp(2\varepsilon_n \tilde S_{j-1}),\ 
|\tilde R_n(t)|\le C\Big( 
\varepsilon_n^2
\sum_{j=1}^n \exp(4\varepsilon_n \tilde S_{j-1})
x^2 \varepsilon_n^2
\Big),
\end{eqnarray}
where the constant $C$ in \textit{absolute} constant, which in particular does not depend on $x$, $t$ or $n$. 
Now, we derive from \eqref{NEW_EQ_FOR_REMAINDERS} that $x^2\varepsilon_n^2\exp\big(2\varepsilon_n \tilde S_{j-1}\big)\le Cn^{-2\gamma}\exp\big(2ct^{\frac 12}(\ln(n))^{\frac 12}\big) \underset{n}{\rightarrow}0$.
Thus,  $|\tilde R_n(t)|\le C\left(\varepsilon_n^2
\sum_{j=1}^n \exp(4\varepsilon_n \tilde S_{j-1})x^2\varepsilon_n^2\right)\le C\tilde A_n(t) n^{-2\gamma}\exp\big(2ct^{\frac 12}(\ln(n))^{\frac 12}\big):=C\tilde A_n(t) \beta_n,\ \beta_n \rightarrow_n 0 $. We get that:
\begin{equation}
\label{ASYMP_PREAL}
 I_n\sim_n \frac 1{2\pi} \int \hat g(\delta_n x) \E\Bigg[\exp\Bigg(-\frac{x^2}{2} \tilde A_n(t)\Bigg)\I_{\widetilde M_n^+\le cn^{\frac 12 }\ln(n)^{\frac 12}}\Bigg]dx.
\end{equation}
Indeed, for all $\lambda\in [0,1] $,
\begin{equation}
\label{BD_REST}
\widetilde A_n(t)+\lambda \tilde R_n(t) \ge 
\varepsilon_n^{2}\sum_{j=1}^n\exp(2\varepsilon_n \tilde S_{j-1})-|\tilde R_n(t)| 
\ge 
\frac12 
\varepsilon_n^2\sum_{j=1}^n\exp(2\varepsilon_n \tilde S_{j-1})=\frac 12\tilde A_n(t).
\end{equation}
so that:
\begin{eqnarray*}
|\Delta_n(t,x)|:=\Bigg|\exp\Bigg(-\frac{x^2}{2} (\tilde A_n(t)+\tilde R_n(t))\Bigg)-\exp\Bigg(-\frac{x^2}{2} \tilde A_n(t)\Bigg)\Bigg|\le \int_0^1 \exp\Bigg(-\frac{x^2}{2} (\tilde A_n(t)+\lambda\tilde R_n(t))\Bigg) \frac{x^2}2 |\tilde R_n(t)| d\lambda,\\
\underset{\eqref{BD_REST}}{\le}  C\exp\Bigg(-\frac{x^2}{4} \tilde A_n(t)\Bigg) \frac{x^2}2 \tilde A_n(t) \beta_n \le C\exp\Bigg(-\frac{x^2}{8} \tilde A_n(t)\Bigg) \beta_n.
\end{eqnarray*}
Thus, exploiting that $\hat g $ is bounded we get:
\begin{eqnarray}\label{PREAL_PREAL}
\Big|\int \hat g(\delta_n x) \E\Big[\Delta_n(t,x)\I_{\widetilde M_n^+\le cn^{\frac 12 }\ln(n)^{\frac 12}}\Big]dx\Big|\le C\beta_n \E[|\tilde A_n(t)|^{-1/2}].
\end{eqnarray}
We now state a useful Proposition, whose proof is postponed to the end of the section for the sake of clarity.
\begin{proposition}\label{MOM_NEG}
For  $\theta \in \{\frac 12, 1\} $ and $n\ge 1 $, there exists $C\ge 1$ s.t.: 
\begin{equation}
\label{CTR_EST}
 \E[|\widetilde A_n(t)|^{-\theta}
 ]\le Ct^{-1}. 
 \end{equation}
\end{proposition}
Let us now prove that Proposition \ref{MOM_NEG} and \eqref{PREAL_PREAL} yield \eqref{ASYMP_PREAL}.

We first split the term $I_n$ introduced in \eqref{DEF_IN} and equivalent to the r.h.s. of \eqref{ASYMP_PREAL} into two parts.
\begin{eqnarray*}
I_n^1&:=&\frac{1}{2\pi}\int_{|x|\le \frac{1}{\sqrt {\delta_n}}} \hat g(\delta_n x)\E[\exp(-\frac{1}{2}x^2 \tilde A_n(t))
\I_{\widetilde M_n^+\le c(n\ln(n))^{\frac 12}}] dx\\
&\sim_n&\hat g(0)\frac{1}{2\pi}\int_{|x|\le \frac{1}{\sqrt {\delta_n}}} \E[\exp(-\frac{1}{2}x^2 \tilde A_n(t))
\I_{\widetilde M_n^+\le c(n\ln(n))^{\frac 12}}] dx=:\bar I_n.\\
\end{eqnarray*}
Now, from the Fubini theorem, we get:
\begin{eqnarray*}
\bar I_n&:=&\hat g(0)\frac{1}{2\pi}\Bigg(  \E\Bigg[\Big\{\int_\R \exp(-\frac{1}{2}x^2 \tilde A_n(t)) dx\Big\}\ 
\I_{\widetilde M_n^+\le c(n\ln(n))^{\frac 12}}\Bigg]
 -\int_{|x| >\frac{1}{\sqrt {\delta_n}}} \E[\exp(-\frac{1}{2}x^2 \tilde A_n(t))
 \I_{\widetilde M_n^+\le c(n\ln(n))^{\frac 12}}] dx\Bigg)\\
 &=& \Bigg(\E[\frac{1}{\sqrt{2\pi \tilde A_n(t)}}
 \I_{\widetilde M_n^+\le c(n\ln(n))^{\frac 12}}] \Bigg)
+ O\Bigg(
\E[\exp(-\frac 14 \frac{\tilde A_n(t)}{\delta_n})\frac 1{\tilde A_n(t)^{1/2}}]
\Bigg)\\
 &=& \Bigg(\E[\frac{1}{\sqrt{2\pi \tilde A_n(t)}}
 \I_{\widetilde M_n^+\le c(n\ln(n))^{\frac 12}}] \Bigg)+O\Bigg(\delta_n^{1/2} \E[(\tilde A_n(t))^{-1}]\Bigg).
\end{eqnarray*}
From Propositions \ref{DP} and \ref{MOM_NEG} and Fatou's lemma, we obtain:
\begin{eqnarray}
\label{ASYMP_MAIN_TERM}
\bar I_n
&\sim_n &\E\Bigg[\frac{1}{\sqrt{2\pi \tilde A(t)}} \Bigg]=p_2(t,0)\sim_{t\rightarrow+\infty}\frac{\pi}{t},
\end{eqnarray}
where $\tilde A(t)=\int_0^t \exp(2\tilde B_s^1)ds$ and $p_2(t,.)$ stands for the density of $\int_0^t \exp(\tilde B_s^1)dB_s^2 $ at time $t$ and point $0$ (see \eqref{return_proba}). Indeed, conditionally to $\{(\tilde B_s^1)_{s\in[0,t]}\} $ the law of $\int_0^t \exp(\tilde B_s^1)dB_s^2 $ is a centered Gaussian with variance $\tilde A(t) $ (Wiener integral). The last equivalence in \eqref{ASYMP_MAIN_TERM} can be derived directly from Proposition 6.6 in \cite{yor:mats:05}. Another derivation, exploiting the explicit large time behavior of the return probability on $\Aff(\R) $ given in Theorem \ref{THM_ASYMP_DENS_GROUP}, is proposed in equation \eqref{CTR_THETA_12_BB} below.
The term $I_n^1 \sim_n \bar I_n$ is the main contribution of $I_n$. The other contribution is small and can be treated as the above remainder. Let us write:
\begin{eqnarray*}
|I_n^2|&:=&\frac{1}{2\pi}\int_{|x|> \frac{1}{\sqrt {\delta_n}}} |\hat g(\delta_n x)|\E[\exp(-\frac{1}{2}x^2 \tilde A_n(t))
\I_{M_n^+\le c(n\ln(n))^{\frac 12}}]dx\\
&\le& C\E[\exp(-\frac 14 \frac{\tilde A_n(t)}{\delta_n})\frac 1{\tilde A_n(t)^{1/2}}]
\le  \delta_n^{1/2} \E[(\tilde A_n(t))^{-1}].
\end{eqnarray*}
This completes the proof of Theorem \ref{LLT_1}.
\end{proof}

\begin{proof}[Proof of Proposition \ref{MOM_NEG}]
We recall from Donati-Martin \textit{et al.} \cite{dona:mats:yor:00} (see also Chaumont \textit{et al.} \cite{chau:hobs:yor:01}) that for a standard Brownian bridge $(b_u)_{u\in [0,1]} $ on $[0,1] $, it holds that for $\alpha  \in \R^{+}$,
\begin{equation}\label{MOM_NEG_BB}
\E\left[\left(\int_0^1\exp(\alpha b_u)du\right)^{-1}\right]=1.
\end{equation}
We now detail how the indicated convergence rate in time can be deduced for $\theta=1 $ and the limit Brownian bridge from \eqref{MOM_NEG_BB}.
Recall that if $(\tilde B_u)_{u\in [0,t]} $ is a standard Brownian bridge on $[0,t] $, then
$$ (\tilde B_u)_{u\in [0,t]}\overset{({\rm law})}{=}\Big( (t-u)\int_0^u \frac{dB_v}{t-v}\Big)_{u\in [0,t]},$$
where $(B_u)_{u\ge 0} $ is a standard Brownian motion. Hence:
\begin{eqnarray}
\E\left[\left(\int_0^t\exp(2 \tilde B_u)du\right)^{-1}\right]&=&\E\left[\left(\int_0^t\exp\Big(2 (t-u)\int_0^{u}\frac{dB_v}{t-v}\Big)du\right)^{-1}\right]\notag\\
&=&t^{-1}\E\left[\left(\int_0^1\exp\Big(2 t(1-u)\int_0^{ut}\frac{dB_v}{t-v}\Big)du\right)^{-1}\right].\label{PREAL_ID_LAW}
\end{eqnarray}
A usual covariance computation then shows that $\Big((1-u)\int_0^{ut}\frac{dB_v}{t-v}\Big)_{u\in [0,1]}\overset{({\rm law} )}{=}\frac{1}{t^{1/2}}\Big((1-u)\int_0^u \frac{dB_v}{1-v}\Big)_{u\in [0,1]}\overset{({\rm law} )}{=}\frac{1}{t^{1/2}}(b_u)_{u\in [0,1]} $. Thus, from \eqref{PREAL_ID_LAW} and \eqref{MOM_NEG_BB}:
\begin{equation}\label{CTR_THETA_1_BB}
\E\left[\left(\int_0^t\exp(2 \tilde B_u)du\right)^{-1}\right]=t^{-1}\E\left[\left(\int_0^1\exp\Big(2 t^{1/2}b_u\Big)du\right)^{-1}\right]=t^{-1}.
\end{equation}
On the other hand, recall that:
\begin{eqnarray*}
p_{\Aff(\R)}(t,e,e)&=&p_{\big(B_t^1,\int_0^t \exp(B_s^1)dB_s^2\big)}(0,0)=p_{B_t^1}(0)p_{\int_{0}^ t\exp(B_s^1)dB_s^2}(0|B_t^1=0)\\
&=&\frac{1}{\sqrt{2\pi t}}\E\left[\left(2\pi\int_0^t \exp(2 B_s^1)ds\right)^{-1/2}\Big|B_t^1=0\right].
\end{eqnarray*}
Hence, the asymptotic behavior of the return density for the Brownian motion on the group given in Theorem \ref{THM_ASYMP_DENS_GROUP} (see also \eqref{return_proba}) yields:
\begin{equation}
\label{CTR_THETA_12_BB}
\E\left[\left(2\pi \int_0^t\exp(2 \tilde B_u^1)du\right)^{-1/2}\right]\sim_{t\rightarrow+\infty}\frac{\pi}{t}.
\end{equation}
Let us now detail how the statement \eqref{CTR_EST} of Proposition \ref{MOM_NEG} can be derived from the previous controls \eqref{CTR_THETA_12_BB}, \eqref{CTR_THETA_1_BB} on the continuous objects through convergence in law arguments.  Starting from our simple random walk $S_0=0, S_k=\sum_{j=1}^k X_j,\ k\ge 1 $ we first introduce for any fixed $n\in \N$ the random polygonal function
$$x_n(u):=S_{\lfloor nu\rfloor }+ (nu-\lfloor nu\rfloor)X_{\lfloor nu\rfloor},\ u\in [0,1],$$
where we recall that $\lfloor \cdot \rfloor $ stands for the integer part. Introducing the rescaled conditioned process $\big(\theta_n(u)\big)_{u\in [0,1]} := \frac 1{\sqrt n}\big(x_n(u)|S_n=0\big)_{u\in [0,1]}$,  we derive from Theorem 2 in \cite{verv:79} 
that $\big(\theta_n(u)\big)_{u\in [0,1]} \Rightarrow \big(b_u\big)_{u\in [0,1]} $, standard Brownian bridge on $[0,1] $ with canonical measure $\mu $ on $C([0,1]) $. Considering now the stepwise constant approximation:
$$\tilde x_n(u):=S_{\lfloor nu\rfloor },\ u\in [0,1],$$
and its associated rescaled conditioned process $\big(\tilde \theta_n(u)\big)_{u\in [0,1]} := \frac 1{\sqrt n}\big(\tilde x_n(u)|S_n=0\big)_{u\in [0,1]} $, it is easily seen that the corresponding measures $\tilde \mu_n  $ on $D([0,1])$ converge weakly in $D[0,1] $ to the distribution $\mu$ (canonical measure of the Brownian bridge on $C([0,1]) $). From the definition of $\tilde A_n(t) $ in \eqref{DECOMP_AN}, recalling a well that $\varepsilon_n=\left( \frac tn\right)^{1/2} $, we thus rewrite:
\begin{eqnarray*}
\tilde A_n(t):=
\varepsilon_n^2\sum_{j=1}^n\exp(2\varepsilon_n \tilde S_{j-1})=\frac tn \sum_{j=1}^n \exp\left(2t^{1/2}\frac{1}{\sqrt{n}}\tilde x_n\big(\frac jn\big)\right)=t\int_0^1 \exp\left(2t^{1/2} \tilde \theta_n(u)\right)du.
\end{eqnarray*}
Hence,  from the previous convergence in law $\tilde A_n(t) \overset{({\rm law})}{\rightarrow} t\int_0^1\exp(2t^{1/2}b_u) du\overset{({\rm law})}{=}\int_0^t \exp(2 \tilde B_u) du$ and for a given $A>0$ and $\theta\in \{\frac 12, 1\} $: 
\begin{eqnarray*}
\E[(\tilde A_n(t))^{-\theta}\I_{A^{-1}\le \tilde A_n(t)\le A}] \underset{n}{\longrightarrow}\E[(\tilde A_t)^{-\theta}\I_{A^{-1}\le \tilde A(t)\le A}].
\end{eqnarray*}
The statement \eqref{CTR_EST} now follows from the above equation and the previously established estimates \eqref{CTR_THETA_12_BB}, \eqref{CTR_THETA_1_BB}, noting as well that, since $\tilde A_n(t)^{-\theta}\le ( t \exp(2t^{1/2}\frac{M_n^-}{\sqrt n}) )^{-\theta} $, Lemma \ref{LEMME_BRIDGE_WALK} gives that the sequence $\big(\tilde A_n(t)^{-\theta} \big)_{n\ge 0}$ is bounded in $L^2(\P)$ and therefore uniformly integrable.
 The proof is complete.
\end{proof}

\begin{remark}[Balance of $n$ and $t$ for the approximation]
Observe from the previous proof of Theorem \ref{LLT_1} that one can actually consider a the same time $n$ and $t$ going to infinity with a suitable polynomial dependence to get a convergent approximation.
\end{remark}


\subsection{The Mixed Case.}
We consider in this Section that the random variables $(Y_i)_{i\in \N} $ in the definition of the random walk approximation \eqref{DEF_MARCHE} are i.i.d. and have common standard Gaussian law, i.e. $Y_i \overset{({\rm law})}{=}{\mathcal N}(0,1) $. This modification is precisely enough to restore the ``expected" local limit theorem.
\begin{theorem}\label{LLT_2}
For the previously described random walk,  taking $\varepsilon_n=\left(\frac{t}{ n}\right)^{\frac 12} $ and for $ n\in 2\N$:
$$\P[a_{\varepsilon_n}=1, b_{\varepsilon_n}\in [0,dx)] =\P[ {S_{n}=0} , \varepsilon_n \sum_{j=1}^{n} Y_j \exp(\varepsilon_n S_{j-1}) \in [0,dx)] \sim_n 2\varepsilon_n  
\cdot p_{\Aff(\R)}(t,e,e)dx .$$
\end{theorem}
We indeed have a result similar to Theorem \ref{LLT_1}, except that no integration with respect to the previous mollifyer $g_{\delta_n} $ is needed.

\begin{proof}
Note that the random variable $b_\varepsilon(n)$ now has a conditional Gaussian density (for fixed trajectory $(S_k)_{k\in \N} $). We thus readily get:
\begin{eqnarray*}
\P[a_{\varepsilon_n}=1, b_{\varepsilon_n}\in [0,dx)] \sim_n \frac{2}{\sqrt{2\pi n}}\E\left[\left.\frac{1}{\sqrt{2\pi \varepsilon_n^2 \sum_{j=1}^n \exp(2\varepsilon_n S_{j-1})}} \right|S_n=0\right]dx.
\end{eqnarray*}
Proposition \ref{MOM_NEG} remains valid taking $\tilde R_n =0$ in the definition \eqref{BD_REST}. With the notations used therein, this precisely gives $\tilde A_n(t)=\varepsilon_n^2 \sum_{j=1}^n \exp(2\varepsilon_n S_{j-1}) $. We then derive the statement from Propositions \ref{DP}, \ref{MOM_NEG}
  and Fatou's lemma.

\end{proof}

\bibliographystyle{alpha}
\bibliography{bibli}

\end{document}